\documentclass[12pt]{article}
\usepackage{amsthm}
\usepackage{amsmath}
\usepackage{enumerate}
\usepackage{amssymb}
\usepackage{latexsym}
\usepackage{amsfonts}
\usepackage{color}
\usepackage{mathrsfs}
\usepackage{epsfig}
\usepackage{url}
\usepackage{xcolor,soul}
\newtheorem{theorem}{\bf Theorem}[section]

\newtheorem{definition}[theorem]{\bf Definition}

\newtheorem{remark}[theorem]{\bf Remark}
\newtheorem{lemma}[theorem]{\bf Lemma}

\newsavebox{\savepar}

\begin{document}
	\title{Existence of positive solutions for a singular elliptic problem with critical exponent and
		measure data}
		\author{Akasmika Panda$^\ddagger$, Debajyoti Choudhuri$^\ddagger$, Ratan Kr. Giri$^{\dagger }$\footnote{Corresponding
				author} \\
			\small{$^\ddagger$Department of Mathematics, National Institute of Technology Rourkela}\\
			\small{Rourkela - 769008, India}\\
			\small{$^\dagger$ Mathematics department, Technion - Israel Institute of Technology,}\\ \small{Amado building, Haifa 32000, Israel
			}\\
			\small{Emails: akasmika44@gmail.com, dc.iit12@gmail.com, giri90ratan@gmail.com}			
		}
		\date{\today}
		\maketitle
		\begin{abstract}
\noindent We prove the existence of a positive {\it SOLA (Solutions Obtained as Limits of Approximations)} to the following PDE involving fractional power of Laplacian
			\begin{equation}
			\begin{split}
			(-\Delta)^su&= \frac{1}{u^\gamma}+\lambda u^{2_s^*-1}+\mu ~\text{in}~\Omega,\\
			u&>0~\text{in}~\Omega,\\
			u&= 0~\text{in}~\mathbb{R}^N\setminus\Omega.
			\end{split}
			\end{equation}
Here, $\Omega$ is a bounded domain of $\mathbb{R}^N$, $s\in (0,1)$, $2s<N$, $\lambda,\gamma\in (0,1)$, $2_s^*=\frac{2N}{N-2s}$ is the fractional critical Sobolev exponent and $\mu$ is a nonnegative bounded Radon measure in $\Omega$.\\
			\textbf{Keywords:} Fractional Sobolev spaces, SOLA, Radon measure, Marcinkiewicz space, critical exponent.\\
			\textbf{2010 AMS Classification:}  35J60, 35R11, 35A15.
		\end{abstract}
		\section{Introduction}
	In this paper, we discuss the following fractional elliptic problem with a singularity, a critical exponent and a Radon measure.
	\begin{equation}\label{4p1}\tag{$P_\lambda$}
			\begin{split}
			(-\Delta)^su&= \frac{1}{u^\gamma}+\lambda u^{2_s^*-1}+\mu ~\text{in}~\Omega,\\
			u&>0~\text{in}~\Omega,\\
			u&= 0~\text{in}~\mathbb{R}^N\setminus\Omega,
			\end{split}
			\end{equation}
			where $\Omega$ is a bounded domain in $\mathbb{R}^N$ with $C^2$ boundary, $s\in(0,1)$, $N>2s$, $0<\gamma<1$, $\lambda>0$, $\mu$ is a nonnegative bounded Radon measure and $(-\Delta)^s$ is the fractional Laplacian defined by 
			$$(-\Delta)^su=\text{P. V.}\int_{\mathbb{R}^N}\frac{u(x)-u(y)}{|x-y|^{N+2s}}dy.$$
Problems involving nonlocal operators have theoretical applications as well as real life applications in various fields of science. The applications of fractional order Laplacian can be found in L\'{e}vy stable diffusion process, chemical reactions in liquids, geophysical fluid dynamics, electromagnetism etc (refer \cite{Applebaum} for further details). Nonlocal problems containing singular or irregular data are used in dislocation problems \cite{Cacase}, quasi-geostrophic dynamics \cite{Caffarelli}, image reconstruction problems \cite{Gilbao} etc. The problem of denoising a image is to find a clear image $u$ from a noisy $f$. In the deblurring problem, a given image $f$ is considered as a blurry version of an unknown exact image $u$, which is to be determined. For further details refer Kinermann et al. \cite{kinder}. Readers may refer to the work in \cite{toma}, \cite{yaro1}, \cite{yaro2}.\\
In general, the presence of a measure data in the problem weakens the class of solution space, i.e. we lose some degrees of differentiability or/and integrability of the solution space. Solutions to problems involving measure data or $L^1$ data are obtained by approximations and usually by working in Marcinkiewicz spaces. Readers may refer \cite{Benilian}, \cite{Boccardo 1},\cite{Boccardo 2},\cite{Kuusi} and the references therein for further readings on these types of problems. Boccardo et al. (\cite{Boccardo 1},\cite{Boccardo 2}) proved that the solution to a nonlinear elliptic equation involving $p$-Laplacian and a Radon measure lies in $W^{1,q}_0(\Omega)$ for every $q<\frac{N(p-1)}{N-1}$, where $1<p<N$. Recently, in 2015, Kuusi et al. \cite{Kuusi} considered the problem, as in \cite{Boccardo,Boccardo 1}, in a fractional set up with fractional $p$-Laplacian and established the existence of a solution in $W^{s_1,q}(\Omega)$ for every $s_1<s<1$, $q<\min\{\frac{N(p-1)}{N-s},p\}$. Purely singular problems both in the local and nonlocal cases are studied in \cite{Boccardo},\cite{Canino},\cite{Lazer}, etc. and the references therein. In all these articles the choice of a solution space depends on the power $\gamma$ of the singular term (whether $\gamma\leq1$ or $\gamma>1$). Further, we refer \cite{Bisci 1},\cite{Bisci 2},\cite{Servadei}, etc. to survey Brezis$-$Nirenberg type critical exponent problems (without the singular term and measure data).\\
The problem $(P_\lambda)$ for $\lambda=0$ and the limiting case of $s=1$ has been analyzed by Panda et al. in \cite{Panda}. The authors have guaranteed the existence of a weak solution in $W_0^{1,q}(\Omega)$ if $0<\gamma\leq1$  and in $W_{loc}^{1,q}(\Omega)$ if $\gamma>1$ for every $q<\frac{N}{N-1}$. Ghosh et al. in \cite{Ghosh} extended this result and studied the problem $(P_\lambda)$ with $s\in(0,1)$ and $\lambda=0$. \\
%In the last few decades, the problem involving a singularity and a critical exponent has been studied by many researchers, both in the local and the nonlocal setup.
%%	\begin{equation}\label{int1}
%%\begin{split}
%%(-\Delta)^su&= \frac{\lambda_1 f(x)}{u^\gamma}+\lambda_2 u^{r} ~\text{in}~\Omega,\\
%%u&>0~\text{in}~\Omega,\\
%%u&= 0~\text{in}~\mathbb{R}^N\setminus\Omega,
%%\end{split}
%%\end{equation}
%%where $\lambda_1,\lambda_2>0$, $1<r\leq 2_s^*-1$, $\gamma>0$ and $f>0$ is a bounded function. 
%We direct the interested readers to refer \cite{Dhanya,Giacomoni 2,Giacomoni,Haitao,Hirano 2,Mukharjee} and the references therein. 
%%The case with $r=2_s^*-1$, $\lambda_2=\lambda$ and $f\equiv\frac{1}{\lambda_1}$ gives the problem ($P_\lambda$) for $\mu\equiv0$. 
%The authors in all these articles have shown the existence and multiplicity of solutions  using different techniques like variational method, concentration compactness method, Perron method and Nehari manifold method.\\
We use the relation among the fractional Sobolev space, Bessel potential space, Marcinkiewicz space to find a solution in a function space weaker than $H_0^s(\Omega)$. Such solutions are called as SOLA (see Definition $\ref{SOLA}$). Due to the presence of nonlinearities with a critical exponent, singularity and a measure data, difficulties arise in the study of $(P_\lambda)$. Thus, it is not easy to directly approach the problem with any commonly used tools like variational method, Nehari manifold method, etc. We study our main problem via a sequence of approximating problems. It is very challenging to prove the existence of a solution to the approximating problems and simultaneously showing the boundedness of the sequence of solution to these approximating problems in $L^{2_s^*}(\Omega)$. To overcome these difficulties we take the help of two auxiliary problems. For that, we apply the concentration compactness principle as in \cite{Mosconi} and Ekeland's variational principle as in \cite{Ekeland}. Precisely, we guarantee that the approximating problem admits at least one solution in a complete Hilbert manifold $H=\{u\in H_0^s(\Omega):\|u\|_{L^{2_s^*}(\Omega)}=1 \}$. We follow some of the arguments of \cite{Panda} to prove our main result stated in the following theorem.
	\begin{theorem}\label{main theorem}
	There exists $0<\Lambda<\infty$ such that for $\lambda\in(0,\Lambda)$ the problem $(P_\lambda)$ admits a positive SOLA $u\in W_{0}^{s_1,q}(\Omega)$ for every $s_1<s$ and $q<\frac{N}{N-s}$ in the sense of  Definition $\ref{SOLA}$.	
\end{theorem}
\noindent Before ending this section we describe the {organization} of the paper. In Section $\ref{2}$, we introduce suitable function spaces to deal with our problem and also provide some auxiliary results which will play important roles throughout the article. In Section $\ref{sec 2}$, we prove the existence of a weak solution to the approximating problem for a certain range of $\lambda$. Section $\ref{sec 3}$ is devoted to the proof of Theorem $\ref{main theorem}$. Further, in the Appendix, we show the multiplicity of solutions {using the Nehari manifold.}
\section{Functional settings and auxiliary results}\label{2}
Let $\Omega$ be a bounded domain in $\mathbb{R}^N$, $1\leq p<\infty$ and $s\in (0,1)$. The fractional order Sobolev space (refer \cite{Valdinocci}) is defined as
$$W^{s,p}(\mathbb{R}^N)=\left\{u\in L^p(\mathbb{R}^N):\int_{\mathbb{R}^{2N}}\frac{{|u(x)-u(y)|^p}}{|x-y|^{N+sp}}dxdy<\infty\right\}$$ and $W_0^{s,p}(\Omega)$ is a subspace of $W^{s,p}(\mathbb{R}^N)$ given by 
$$W_0^{s,p}(\Omega)=\{u\in W^{s,p}(\mathbb{R}^N): \int_{\mathbb{R}^{2N}}\frac{|u(x)-u(y)|^{p}}{|x-y|^{N+sp}}dydx<\infty,~u=0 \text{ in }\mathbb{R}^N\setminus\Omega\}$$ equipped with the norm $$\|u\|_{W_0^{s,p}(\Omega)}^p=\int_{\mathbb{R}^{2N}}\frac{|u(x)-u(y)|^{p}}{|x-y|^{N+sp}}dydx.$$
\noindent Further, the space $(W_0^{s,p}(\Omega),\|\cdot\|_{W_0^{s,p}(\Omega)})$ with $p>1$ is a reflexive separable Banach space. The following classical theorem will be used frequently in this article.
\begin{theorem}[Theorem 6.5, \cite{Valdinocci}]\label{constant}
	Let $0<s<1$ and $p\in[1,\infty)$ with $sp<N$. Then there exists constant $C=C(N,s,p)>0$ such that for any $u\in W_0^{s,p}(\Omega)$, 
	$$\|u\|_{L^r(\Omega)}\leq C \|u\|_{W_0^{s,p}(\Omega)} $$
	for any $r\in[p,p_s^*]$, where $p_s^*=\frac{Np}{N-sp}$. Moreover, the space $W^{s,p}_0(\Omega)$ is continuously embedded in $L^r(\Omega)$ for every $r\in[1,p_s^*]$ and compactly embedded in $L^r(\Omega)$ for every $r\in[1,p_s^*)$.
\end{theorem}
\noindent Denote 
\begin{equation}\label{best}
\mathbb{S}_{s,p}=\underset{u\in W_0^{s,p}(\Omega)\setminus\{0\}}{\inf}\frac{\|u\|_{W_0^{s,p}(\Omega)}^p}{\|u\|_{L^{p_s^*}(\Omega)}^p}
\end{equation}
which is the best Sobolev constant in the Sobolev embedding (Theorem $\ref{constant}$). We now define some function spaces which will be further used in this article.
\begin{remark}
For $p=2$, we denote the Sobolev space $H^s(\mathbb{R}^N)=W^{s,2}(\mathbb{R}^N)$. These spaces are Hilbert spaces. Proposition 3.6 of \cite{Valdinocci} provides the relationship between the fractional Sobolev space $H^s(\mathbb{R}^N)$ and the fractional Laplacian $(-\Delta)^s$. It states that the norms $\|\cdot\|_{H^s(\mathbb{R}^N)}$ and $\|(-\Delta)^{s/2}\cdot\|_{L^2(\mathbb{R}^N)}$ are two equivalent norms.
\end{remark}
\begin{definition}[\cite{Shieh}]\label{Bessel}
For $s\in (0,1)$ and $p\in [1,\infty)$, the Bessel potential space $L^{s,p}(\mathbb{R}^N)$ is defined as
$$L^{s,p}(\mathbb{R}^N)=\{u\in L^{p}(\mathbb{R}^N):|\nabla^s u|\in L^{p}(\mathbb{R}^N)\}$$
where $\nabla^s u=\int_{\mathbb{R}^N}\frac{u(x)-u(y)}{|x-y|^{N+s}}\frac{x-y}{|x-y|}dy$ is the fractional gradient of order $s$.
\end{definition}
\noindent We refer Theorem 2.2 of \cite{Shieh} to see the relation between the fractional Sobolev spaces and the Bessel potential spaces.
\begin{theorem}\label{properties}
\begin{enumerate}
	\item For $s\in(0,1)$ and $p=2$, $L^{s,2}(\mathbb{R}^N)=W^{s,2}(\mathbb{R}^N)$.
	\item For $s\in (0,1)$, $1<p<\infty$ and $0<\epsilon<s$, the following continuous embedding holds
	$$L^{s+\epsilon,p}(\mathbb{R}^N)\subset W^{s,p}(\mathbb{R}^N)\subset L^{s-\epsilon,p}(\mathbb{R}^N).$$ 
\end{enumerate}
\end{theorem}
\begin{definition}
A measurable function $u:\Omega\rightarrow\mathbb{R}$ is said to be in the Marcinkiewicz space $M^q(\Omega)$ $(0<q<\infty)$ if $$m(\{x\in\Omega:|u(x)|>t\})\leq \frac{C}{t^q},~~\text{for }t>0 \text{ and }0<C<\infty.$$
\end{definition}
\begin{remark}
For $\Omega$ bounded, 
\begin{enumerate}
	\item $M^{q_1}(\Omega)\subset M^{q_2}(\Omega)$ for every $q_1\geq q_2>0$.
	\item For $1\leq q<\infty$ and $0<\epsilon<q-1$, the following continuous embedding holds
	\begin{equation}\label{continuous}
	L^q(\Omega)\subset 	M^q(\Omega)\subset 	L^{q-\epsilon}(\Omega).
	\end{equation}
\end{enumerate}
\end{remark}
\noindent For a fixed $k>0$, we denote the truncation functions $T_k:\mathbb{R}\rightarrow \mathbb{R}$ by
\begin{eqnarray}
\begin{split}
T_k(s)=&\begin{cases}
s&\text{if}~ |s|\leq k\\ 
k~sign{s} &\text{if}~ |s|>k.\nonumber 
\end{cases}
\end{split}
\end{eqnarray}
\noindent Since our problem, defined in $(P_\lambda)$, involves a measure data as a nonhomogeneous term in the right hand side, we need to introduce the notion of convergence in measure.
\begin{definition}\label{measure}
	Let $\mathcal{M}(\Omega)$ be the set of all finite Radon measures on $\Omega$ and $(\mu_n)$ be a sequence of measurable functions in $\mathcal{M}(\Omega)$. Then we say $(\mu_n)$ converges to $\mu\in \mathcal{M}(\Omega)$ in the sense of measure if $$\int_{\Omega}\mu_n \phi\rightarrow\int_\Omega \phi d\mu,~~\forall \phi \in C_0(\bar{\Omega}).$$  
\end{definition}
\noindent In the following theorem we state a commonly used variational principle, introduced by Ekeland in \cite{Ekeland}. Ekeland variational principle is also used to show multiplicity.
	\begin{theorem}\label{eke}
	(Ekeland Variational Principle $\cite{Ekeland}$) Let $V$ be a Banach space and $\Psi:V\rightarrow\mathbb{R}\cup\{+\infty\}$ is a lower semicontinuous,  G$\hat{a}$teaux-differentiable and bounded from below function. Then for every $\epsilon>0$, every $u\in V$ satisfying $ \Psi(u)\leq\inf\Psi+\epsilon$, every $\delta>0$, there exists $v\in V$ such that $\Psi(v)\leq\Psi(u)$, $\|u-v\|\leq\delta$ and $\|\Psi^\prime(v)\|^*\leq\frac{\epsilon}{\delta}$
	where $\|.\|$ and $\|.\|^*$ are the norm of $V$ and the dual norm of $V$, respectively.
\end{theorem}	 
\noindent We now introduce a suitable notion of solution to $(P_\lambda)$ that in general do not lie in the natural energy space corresponding to the operator $(-\Delta)^s$, i.e. $H_0^s(\Omega)$, but has a lower degree of differentiability and integrability. They are called SOLA (Solutions Obtained as Limits of Approximations) and the procedure of construction of SOLA is through a sequence of approximating problems.
\begin{definition}[SOLA for ($P_\lambda$)]\label{SOLA}
	Let $\mu\in \mathcal{M}(\Omega)$ and $0<\gamma<1$. Then we say that a function $u\in W_0^{s_1,q}(\Omega)$ for $s_1<s$, $q<\frac{N}{N-s}$, is a SOLA to $(P_\lambda)$ if
	 \begin{equation}\label{SOLA weak}
	 \int_{\mathbb{R}^{N}}(-\Delta)^{s/2}u\cdot(-\Delta)^{s/2}\phi=\int_{\Omega}\frac{1}{u^\gamma}\phi+\int_{\Omega}\lambda u^{2_s^*-1}\phi+\int_{\Omega} \phi d\mu,~~\forall \phi\in C_c^\infty(\Omega)	
	 \end{equation}
	 and for every $\omega\subset \subset\Omega$, there exists a $C_\omega$ such that 
	 \begin{equation}\label{comapct}
	 u\geq C_\omega>0.
	 \end{equation}
Furthermore,  there exists a sequence of weak solutions $(u_n)\subset H_0^s(\Omega)$ to the approximate Dirichlet problems  
\begin{equation}\label{4p2}\tag{$P_{\lambda,n}$}
\begin{split}
(-\Delta)^su_n&= \frac{1}{(u_n+\frac{1}{n})^\gamma}+\lambda u_n^{2_s^*-1}+\mu_n ~\text{in}~\Omega,\\
u_n&>0~\text{in}~\Omega,\\
	u_n&= 0~\text{in}~\mathbb{R}^N\setminus\Omega,
		\end{split}
		\end{equation}
	in the sense of Definition $\ref{weak form of apprpx}$, such that $u_n$ converges to $u$ a.e. in $\mathbb{R}^N$ and locally in $L^q(\mathbb{R}^N)$. Here $(\mu_n)\subset L^\infty(\Omega)$ is a positive $L^1$ bounded sequence which converges to $\mu$ in the sense of measure as defined in Definition $\ref{measure}$.
	\end{definition}
		\begin{definition}\label{weak form of apprpx}
			A function $u_n\in H_0^s(\Omega)$ is said to be a positive weak solution of ($P_{\lambda,n}$) if for every $\phi\in C_c^\infty(\Omega)$
			 \begin{equation}\label{approximating weak}
		\int_{\mathbb{R}^{N}}(-\Delta)^{s/2}u_n\cdot(-\Delta)^{s/2}\phi=\int_{\Omega}\frac{1}{(u_n+\frac{1}{n})^\gamma}\phi+\int_{\Omega}\lambda u_n^{2_s^*-1}\phi+\int_{\Omega}\mu_n \phi	
			\end{equation}
			and for every $\omega\subset \subset\Omega$, there exists a $C_\omega$ such that $u_n\geq C_\omega>0$.
		\end{definition} 
\noindent We begin with the following sequence of problems.
	\begin{equation}\label{4p3}\tag{$P^1_{\lambda,n}$}
	\begin{split}
	(-\Delta)^sw_n&= \frac{1}{(w_n+\frac{1}{n})^\gamma}+\mu_n ~\text{in}~\Omega,\\
	w_n&>0~\text{in}~\Omega,\\
	w_n&= 0~\text{in}~\mathbb{R}^N\setminus\Omega.
	\end{split}
	\end{equation}
Let us define the spaces $H$ and $\bar{H}$ as follows.
\begin{equation}\label{H}
H=\{u\in H_0^s(\Omega):\|u\|_{L^{2_s^*}(\Omega)}=1\}~\text{and}~\bar{H}=\{u\in H_0^s(\Omega):\|u\|_{L^{2_s^*}(\Omega)}<1\}.
\end{equation} 
 We now look for a weak solution to $(P_{\lambda,n}^1)$ in $\bar{H}$. Following the proof of {Lemmata 2.3 and 2.4} of Ghosh et al. \cite{Ghosh}, the problem $(P^1_{\lambda,n})$ admits a positive weak solution $w_n$ in $\bar{H}\cap L^\infty(\Omega)$. Furthermore, for every $n\in\mathbb{N}$ and for every relatively compact set $\omega\subset \Omega$, there exists a constant $C_\omega$ independent of $n$ such that $w_n\geq C_\omega>0$.
\begin{remark}
	The solution to the problem $(P^1_{\lambda,n})$ is unique. To prove this, assume that the problem has two different solutions $w_n$ and $\bar{w}_n$. Let us consider $(w_n-\bar{w}_n)^+$ as a test function in the weak formulation of $(P^1_{\lambda,n})$. Then by \cite{Leonori}, we have
	\begin{align}
	0&\leq\int_{\mathbb{R}^N}|(-\Delta)^{s/2}(w_n-\bar{w}_n)^+|^2\nonumber\\&\leq\int_{\mathbb{R}^N}(-\Delta)^{s/2}(w_n-\bar{w}_n)\cdot(-\Delta)^{s/2}(w_n-\bar{w}_n)^+\nonumber\\&=\int_{\Omega}\Big(\frac{1}{(w_n+\frac{1}{n})^\gamma}-\frac{1}{(\bar{w}_n+\frac{1}{n})^\gamma}\Big)(w_n-\bar{w}_n)^+\nonumber\\&\leq 0.
	\end{align}
	This implies $(w_n-\bar{w}_n)^+=0$ a.e in $\Omega$ and $w_n\leq \bar{w}_n$ a.e in $\Omega$. In a similar manner, taking $(\bar{w}_n-w_n)^+$ as a test function we can show that $w_n\geq \bar{w}_n$ a.e in $\Omega$. This proves the claim.
\end{remark}
\noindent We observe that $w_n+v_n$ is a solution to $(P_{\lambda,n})$ if and only if $w_n$ is a weak solution to $(P^1_{\lambda,n})$ and $v_n$ is a weak solution to the following problem.
	\begin{equation}\label{4p4}\tag{$P^2_{\lambda,n}$}
	\begin{split}
	(-\Delta)^sv_n+\frac{1}{(w_n+\frac{1}{n})^\gamma}-\frac{1}{(v_n+w_n+\frac{1}{n})^\gamma}&=\lambda(w_n+v_n)^{2_s^*-1}  ~\text{in}~\Omega,\\
	v_n&>0~\text{in}~\Omega,\\
	v_n&= 0~\text{in}~\mathbb{R}^N\setminus\Omega.
	\end{split}
	\end{equation}
The following theorem guarantees the existence of a weak solution to $(P^2_{\lambda,n})$ in the set $H_n$ where $H_n$ is defined by $$H_n=\{u\in H_0^s(\Omega):\|u+w_n\|_{L^{2_s^*}(\Omega)}=1\}.$$ 
Clearly, $H_n\subset \bar{H}$, where $\bar{H}$ is given in $\eqref{H}$.
\begin{theorem}\label{second solution}
There exists $\Lambda>0$ such that for $\lambda\in(0,\Lambda)$, the problem $(P^2_{\lambda,n})$ has a positive weak solution $v_n$ in $H_n$. 
\end{theorem}
\noindent We will prove this theorem in Section $\ref{sec 2}$. Denote $u_n=w_n+v_n$, where $w_n\in \bar{H}$ and $v_n\in H_n$ are the positive weak solutions to $(P^1_{\lambda,n})$ and $(P^2_{\lambda,n})$, respectively. Thus, $u_n\in H$ is a weak solution to $(P_{\lambda,n})$ where $H$ is defined in $\eqref{H}$. We are now in a position to state the following existence theorem.
\begin{theorem}\label{approx solution}
	There exists $0<\Lambda<\infty$ such that for $\lambda\in (0,\Lambda)$  the problem $(P_{\lambda,n})$ admits a positive weak solution $u_n\in H$, in the sense of Definition $\ref{weak form of apprpx}$.
\end{theorem}  
\section{Existence of positive weak solution to $(P_{\lambda,n}^2)$ }\label{sec 2}
Define a function $g_n:\Omega\times\mathbb{R}\rightarrow \mathbb{R}\cup\{-\infty\}$ by
\begin{eqnarray}\label{gn}
\begin{split}
g_n(x,t)=&\begin{cases}
\frac{1}{(w_n(x)+\frac{1}{n})^\gamma}-\frac{1}{(t+w_n(x)+\frac{1}{n})^\gamma}&\text{if}~ t+w_n(x)+\frac{1}{n}>0\\ 
-\infty&\text{otherwise}.
\end{cases}
\end{split}
\end{eqnarray}
Clearly, $g_n$ is nonnegative and non decreasing in $t$. The required measurability of $g_n(\cdot,t)$ follows
from Lemmata 1 and 2 of \cite{Hirano 2}.\\
 Denote $G_n(x,t)=\int_{0}^{t}g_n(x,\tau)d\tau$ for $(x,t)\in \Omega\times\mathbb{R}$. We define the energy functional $I_{\lambda,n}:H_0^s(\Omega)\rightarrow (-\infty,\infty]$, corresponding to $(P^2_{\lambda,n})$, by
\begin{eqnarray}
\begin{split}
I_{\lambda,n}(v)=&\begin{cases}
\frac{1}{2}\int_{\mathbb{R}^{2N}}\frac{|v(x)-v(y)|^2}{|x-y|^{N+2s}}dxdy+\int_{\Omega}G_n(x,v)dx-\frac{\lambda}{2_s^*}\int_{\Omega}|v+w_n|^{2_s^*}dx&\text{if}~ G_n(.,v)\in L^1(\Omega)\\ 
\infty&\text{otherwise}.
\end{cases}
\end{split}
\end{eqnarray}
Further, 
$$\langle I^\prime_{\lambda,n}(v),\bar{v}\rangle=\int_{\mathbb{R}^{2N}}\frac{(v(x)-v(y))(\bar{v}(x)-\bar{v}(y)}{|x-y|^{N+2s}}dxdy+\int_{\Omega}g_n(x,v)\bar{v}dx-\lambda\int_{\Omega}|v+w_n|^{2_s^*-1}\bar{v}dx$$
for any $\bar{v}\in H_0^s(\Omega)$. We now define the weak solution of $(P^2_{\lambda,n})$ as follows.
\begin{definition}
A function $v_n\in H_n$ is said to be a weak solution of $(P^2_{\lambda,n})$ if $v_n$ is a critical point of the functional $I_{\lambda,n}$.	
\end{definition}
\begin{lemma}\label{P-S}
The functional $I_{\lambda,n}$ satisfies the Palais-Smale condition in $H_n$ for energy level
$$c<\frac{s}{N}\frac{\mathbb{S}_{2,s}^{\frac{N}{2s}}}{\lambda^{\frac{N-2s}{2s}}}-\frac{\lambda}{2_s^*}.$$
\begin{proof}
Let $(v_{n,m})\subset H_n$ be a Palais-Smale sequence of $I_{\lambda,n}$, i.e. $I_{\lambda,n}(v_{n,m})\rightarrow c$ and $I^\prime_{\lambda,n}(v_{n,m})\rightarrow 0$. Clearly, the functional $I_{\lambda,n}$ is coercive restricted to $H_n$ and hence the sequence $(v_{n,m})$ is bounded in $H_0^s(\Omega)$. Thus, there exists $v_n\in H_0^s(\Omega)$ and a subsequence of $v_{n,m}$, which is still denoted as $(v_{n,m})$, such that $v_{n,m}\rightarrow v_n$ weakly in $H_0^s(\Omega)$.\\
{\it Claim:} $v_{n,m}\rightarrow v_n$ strongly in $H_0^s(\Omega)$ and $v_n\in H_n$.\\
By the concentration compactness principle [Theorem 2.5, \cite{Mosconi}] for the case $p=2$, there exist two positive Borel regular measures $\nu_1,\nu_2$ such that 
\begin{equation}\label{first}
\int_{\mathbb{R}^N}\frac{|v_{n,m}(x)-v_{n,m}(y)|^{2}}{|x-y|^{N+2s}}dy\overset{t}{\rightharpoonup}\nu_1\geq\int_{\mathbb{R}^N}\frac{|v_n(x)-v_n(y)|^{2}}{|x-y|^{N+2s}}dy+\sum_{j\in I}\nu_{1,j}\delta_{x_j},~\nu_{1,j}=\nu_1\{x_j\}, 
\end{equation}
\begin{equation}\label{second}
|v_{n,m}|^{2_s^*}\overset{t}{\rightharpoonup}\nu_2=|v_n|^{2_s^*}+\sum_{j\in I}\nu_{2,j}\delta_{x_j},~\nu_{2,j}=\nu_2\{x_j\}
\end{equation}	and		
\begin{equation}\label{third}
\mathbb{S}_{2,s}\nu_{2,j}^{\frac{2}{2_s^*}}\leq\nu_{1,j}, ~\forall j\in I
\end{equation}
where $\{x_j:j\in I\}$, $I$ countable, is a set of distinct points in $\mathbb{R}^N$,  $\{\nu_{1,j}:j\in I\}\subset(0,\infty),~ \{\nu_{2,j}:j\in I\}\subset (0,\infty)$ and $\mathbb{S}_{2,s}$ is the best Sobolev constant given in  $\eqref{best}$. Here the symbol $\overset{t}{\rightharpoonup}$ denotes the tight convergence. Hence, if $I=\emptyset$ then $v_{n,m}\rightarrow v_n$ strongly in $L^{2_s^*}(\Omega)$ and $v_n\in H_n$.\\
Suppose $I\neq \emptyset$. Then choose $\zeta\in C_c^\infty(\mathbb{R}^N)$ with support in a unit ball of $\mathbb{R}^N$ such that $0\leq\zeta\leq1$ and $\zeta(0)=1$. Let us define for any $\epsilon>0$, the function $\zeta_{\epsilon,j}$ as $\zeta_{\epsilon,j}=\zeta(\frac{x-x_j}{\epsilon})$. Then
\begin{align}\label{1}
\langle I_{\lambda,n}^\prime(v_{n,m}), \zeta_{\epsilon,j}v_{n,m}\rangle&=\int_{\mathbb{R}^{2N}}\frac{(v_{n,m}(x)-v_{n,m}(y))(\zeta_{\epsilon,j}v_{n,m}(x)-\zeta_{\epsilon,j}v_{n,m}(y)}{|x-y|^{N+2s}}dxdy\nonumber\\&~~~+\int_{\Omega}g_n(x,v_{n,m})\zeta_{\epsilon,j}v_{n,m}-\lambda\int_{\Omega}|v_{n,m}+w_n|^{2_s^*-1}\zeta_{\epsilon,j}v_{n,m}\nonumber\\&\geq\int_{\mathbb{R}^{2N}}\frac{(v_{n,m}(x)-v_{n,m}(y))v_{n,m}(y)(\zeta_{\epsilon,j}(x)-\zeta_{\epsilon,j}(y))}{|x-y|^{N+2s}}dxdy\nonumber\\&~~~+\int_{\mathbb{R}^{2N}}\frac{|v_{n,m}(x)-v_{n,m}(y)|^{2}\zeta_{\epsilon,j}(x)}{|x-y|^{N+2s}}dxdy+\int_{\Omega}g_n(x,v_{n,m})\zeta_{\epsilon,j}v_{n,m}\nonumber\\&~~~-\lambda\int_{\Omega}|v_{n,m}+w_n|^{2_s^*}\zeta_{\epsilon,j}dx\nonumber\\&\geq\int_{\mathbb{R}^{2N}}\frac{(v_{n,m}(x)-v_{n,m}(y))v_{n,m}(y)(\zeta_{\epsilon,j}(x)-\zeta_{\epsilon,j}(y))}{|x-y|^{N+2s}}dxdy\nonumber\\&~~~+\int_{\mathbb{R}^{2N}}\frac{|v_{n,m}(x)-v_{n,m}(y)|^{2}\zeta_{\epsilon,j}(x)}{|x-y|^{N+2s}}dxdy+\int_{\Omega}g_n(x,v_{n,m})\zeta_{\epsilon,j}v_{n,m}\nonumber\\&~~~-\frac{\lambda2^*_s}{2}\int_{\Omega}\left(|v_{n,m}+w_n|^{2_s^*-1}w_n+v_{n,m}^{2_s^*-1}w_n\right)\zeta_{\epsilon,j}-\lambda\int_{\Omega}|v_{n,m}|^{2_s^*}\zeta_{\epsilon,j}.
\end{align}
By Mosconi \& Squassina \cite{Mosconi}, we have $\lim\limits_{\epsilon\rightarrow 0}\int_{\mathbb{R}^{2N}}\frac{|v_{n}(y)|^2|\zeta_{\epsilon,j}(x)-\zeta_{\epsilon,j}(y)|^2}{|x-y|^{N+2s}}dxdy=0$. Thus, on using the H\"{o}lder's inequality we have 
\begin{align}\label{use1}
&\lim\limits_{\epsilon\rightarrow 0}\lim\limits_{m\rightarrow\infty}\left|\int_{\mathbb{R}^{2N}}\frac{(v_{n,m}(x)-v_{n,m}(y))v_{n,m}(y)(\zeta_{\epsilon,j}(x)-\zeta_{\epsilon,j}(y))}{|x-y|^{N+2s}}dxdy\right|\nonumber\\&\leq \lim\limits_{\epsilon\rightarrow 0}\lim\limits_{m\rightarrow\infty} \|v_{n,m}\|_{H_0^s(\Omega)}\left(\int_{\mathbb{R}^{2N}}\frac{|v_{n,m}(y)|^2|\zeta_{\epsilon,j}(x)-\zeta_{\epsilon,j}(y)|^2}{|x-y|^{N+2s}}dxdy\right)^{1/2}\nonumber\\&\leq \lim\limits_{\epsilon\rightarrow 0}\left(\int_{\mathbb{R}^{2N}}\frac{|v_{n}(y)|^2|\zeta_{\epsilon,j}(x)-\zeta_{\epsilon,j}(y)|^2}{|x-y|^{N+2s}}dxdy\right)^{1/2}\nonumber\\&=0.
\end{align}
Since $\zeta(0)=1$ and for $x\neq x_j$, $\zeta_{\epsilon,j}(x)\rightarrow 0$ as $\epsilon\rightarrow 0$, by using $\eqref{first}$ and $\eqref{second}$ we have
\begin{align}
\lim\limits_{\epsilon\rightarrow 0}\lim\limits_{m\rightarrow\infty}\int_{\Omega}\left(|v_{n,m}+w_n|^{2_s^*-1}w_n+|v_{n,m}|^{2_s^*-1}w_n\right)\zeta_{\epsilon,j}&=\lim\limits_{\epsilon\rightarrow 0}\int_{\Omega}\left(|v_{n}+w_n|^{2_s^*-1}w_n+|v_{n}|^{2_s^*-1}w_n\right)\zeta_{\epsilon,j}=0,\\ \lim\limits_{\epsilon\rightarrow 0}\lim\limits_{m\rightarrow\infty}\int_{\mathbb{R}^{2N}}\frac{|v_{n,m}(x)-v_{n,m}(y)|^{2}\zeta_{\epsilon,j}(x)}{|x-y|^{N+2s}}dxdy&=\lim\limits_{\epsilon\rightarrow 0}\int_{\mathbb{R}^{N}}\zeta_{\epsilon,j}d\mu=\nu_{1,j},\\\lim\limits_{\epsilon\rightarrow 0}\lim\limits_{m\rightarrow\infty}\int_{\Omega}g_n(x,v_{n,m})\zeta_{\epsilon,j}v_{n,m}&=\lim\limits_{\epsilon\rightarrow 0}\int_{\mathbb{R}^{N}}g_n(x,v_n)v_n\zeta_{\epsilon,j}=0,\\\lim\limits_{\epsilon\rightarrow 0}\lim\limits_{m\rightarrow\infty}\int_{\Omega}|v_{n,m}|^{2_s^*}\zeta_{\epsilon,j}&=\lim\limits_{\epsilon\rightarrow 0}\int_{\Omega}\zeta_{\epsilon,j}d\nu=\nu_{2,j}.\label{use2}
\end{align}
Thus, with the consideration of the equations $\eqref{use1}-\eqref{use2}$, from $\eqref{1}$ we have $0\geq \nu_{1,j}-\lambda\nu_{2,j}$. This further implies that $\nu_{1,j}\leq\lambda\nu_{2,j}$. From $\eqref{third}$, we already have $\mathbb{S}_{2,s}\nu_{2,j}^{\frac{2}{2^*_s}}\leq\nu_{1,j}$. Hence, we obtain $\nu_{2,j}\geq \left(\frac{\mathbb{S}_{2,s}}{\lambda}\right)^{\frac{N}{2s}}$.\\
On the other hand,
\begin{align}
c&=\lim\limits_{m\rightarrow\infty}I_{\lambda,n}(v_{n,m})\nonumber\\&= \lim\limits_{m\rightarrow\infty}\left(\frac{1}{2}\int_{\mathbb{R}^{2N}}\frac{|v_{n,m}(x)-v_{n,m}(y)|^2}{|x-y|^{N+2s}}dxdy+\int_{\Omega}G_n(x,v_{n,m})dx\right)-\frac{\lambda}{2_s^*}\nonumber\\&\geq\frac{1}{2}\int_{\mathbb{R}^{N}}\left(\int_{\mathbb{R}^N}\frac{|v_n(x)-v_n(y)|^{2}}{|x-y|^{N+2s}}dy+\sum_{j\in I}\nu_{1,j}\right) dx-\frac{\lambda}{2_s^*}\nonumber\\&\geq \frac{s}{N}\nu_{1,j}-\frac{\lambda}{2_s^*}\nonumber\\&\geq\frac{s}{N}\mathbb{S}_{2,s}\nu_{2,j}^{\frac{2}{2_s^*}}-\frac{\lambda}{2_s^*}\nonumber\\&\geq \frac{s}{N} \frac{\mathbb{S}_{2,s}^{\frac{N}{2s}}}{\lambda^{\frac{N-2s}{2s}}}-\frac{\lambda}{2_s^*},\nonumber
\end{align}
which is a contradiction to our assumption $c<\frac{s}{N} \frac{\mathbb{S}_{2,s}^{\frac{N}{2s}}}{\lambda^{\frac{N-2s}{2s}}}-\frac{\lambda}{2_s^*}$. Hence, the indexing set $I$ is empty and $v_{n,m}\rightarrow v_n$ strongly in $L^{2_s^*}(\Omega)$ and $v_n\in H_n$. \\
It remains to prove that $v_{n,m}\rightarrow v_n$ strongly in $H_0^s(\Omega)$. We use a standard method to prove this claim. Recall 
\begin{align}\label{conv}
\langle I^\prime_{\lambda,n}(v_{n,m}),v_{n,m}-v_n\rangle&=\int_{\mathbb{R}^{2N}}\frac{(v_{n,m}(x)-v_{n,m}(y))((v_{n,m}-v_n)(x)-(v_{n,m}-v_n)(y))}{|x-y|^{N+2s}}dxdy\nonumber\\&~~~+\int_{\Omega}g_n(x,v_{n,m})(v_{n,m}-v_n)-\lambda\int_{\Omega}|v_{n,m}+w_n|^{2_s^*-1}(v_{n,m}-v_n).
\end{align}
Since $(v_{n,m})$ is a bounded Palais-Smale sequence, on passing the limit $m\rightarrow\infty$ in $\eqref{conv}$ we have 
$$\lim\limits_{m\rightarrow\infty}\int_{\mathbb{R}^{2N}}\frac{(v_{n,m}(x)-v_{n,m}(y))((v_{n,m}-v_n)(x)-(v_{n,m}-v_n)(y))}{|x-y|^{N+2s}}dxdy=0.$$
On using a simple calculation we get
\begin{align}\label{G}
\|v_{n,m}-v_n\|_{H_0^s(\Omega)}^{2}&=\int_{\mathbb{R}^{2N}}\frac{|(v_{n,m}-v_n)(x)-(v_{n,m}-v_n)(y)|^{2}}{|x-y|^{N+2s}}dxdy\nonumber\\&\leq C\int_{\mathbb{R}^{2N}}
\left\{\frac{(v_{n,m}(x)-v_{n,m}(y))((v_{n,m}-v_n)(x)-(v_{n,m}-v_n)(y))}{|x-y|^{N+2s}}\right. \nonumber \\
&~~~~\left.\kern-\nulldelimiterspace -\;\frac{(v_n(x)-v_n(y))((v_{n,m}-v_n)(x)-(v_{n,m}-v_n)(y))}{|x-y|^{N+2s}}\vphantom{}\right\}.\nonumber
\end{align}
Thus, $\lim\limits_{m\rightarrow\infty}\|v_{n,m}-v_n\|_{H_0^s(\Omega)}=0$ and hence $v_{n,m}\rightarrow v_n$ strongly in $H_0^s(\Omega)$. Therefore, $v_n\in H_n$ is a critical point of $I_{\lambda,n}$ and hence a weak solution to $(P^2_{\lambda,n})$. 
\end{proof}	
\end{lemma}
\noindent Consider the sequence $(V_\epsilon)$ which is given by $$V_\epsilon=\epsilon^{-\frac{N-2s}{2}} v^*\left(\frac{x}{\epsilon}\right), ~~x\in \mathbb{R}^N.$$
Here $v^*(x)=\bar{v}\left(\frac{x}{\mathbb{S}_{2,s}^{\frac{1}{2s}}}\right)$, $\bar{v}(x)=\frac{\tilde{v}(x)}{\|\tilde{v}\|_{L^{2_s^*}(\Omega)}}$ and $\tilde{v}(x)=\beta(\alpha^2+|x|^2)^{-\frac{N-2s}{2}}$ with two fixed constants $\beta\in \mathbb{R}^N\setminus\{0\},~\alpha>0$. According to Servadei \& Valdinoci \cite{Sevadei 1}, for each $\epsilon>0$ the corresponding $V_\epsilon$ satisfies the problem
$$(-\Delta)^s v=|v|^{2_s^*-2}v~~\text{ in }\mathbb{R}^N$$ and $$\int_{\mathbb{R}^{2N}}\frac{|V_\epsilon(x)-V_\epsilon(y)|^2}{|x-y|^{N+2s}}dxdy=\int_{\mathbb{R}^N}|V_\epsilon|^{2_s^*}dx=\mathbb{S}_{2,s}^{N/2s}.$$
Without loss of generality we can assume $0\in\Omega$. Consider the function $\zeta\in C_c^{\infty}(\mathbb{R}^N)$ such that $0\leq\zeta\leq 1$ and for a fixed $\delta>0$ with $B_{4\delta}\subset\Omega$, $\zeta\equiv 0$ in $\mathbb{R}^N\setminus B_{2\delta}$, $\zeta\equiv1$ in $B_\delta$. Let us define a function $\Psi_\epsilon(x)=\zeta(x)V_\epsilon(x)$, which is zero in $\mathbb{R}^N\setminus\Omega$. By Giacomoni et al. \cite{Giacomoni}, there exists $a_1,a_2,a_3>0$ such that for $1<q<\min\{2,\frac{N}{N-2s}\}$ we have the following estimates.
\begin{align}
\int_{\mathbb{R}^{2N}}\frac{|\Psi_\epsilon(x)-\Psi_\epsilon(y)|^2}{|x-y|^{N+2s}}dxdy&\leq \mathbb{S}_{2,s}^{N/2s}+a_1\epsilon^{N-2s},\nonumber\\
\int_{\Omega}|\Psi_\epsilon|^{2_s^*}dx&\geq \mathbb{S}_{2,s}^{N/2s}-a_2\epsilon^N,\nonumber\\
\int_{\Omega}|\Psi_\epsilon|^{q}dx&\leq a_3\epsilon^{(N-2s)q/2}.\nonumber
\end{align}
\begin{lemma}\label{upper bound}
There exists $\Lambda>0$ such that for sufficienty small $\epsilon>0$ and for $\lambda\in (0,\Lambda)$,
$$\sup\{I_{\lambda,n}(t\Psi_\epsilon): t\geq 0\}<\frac{s}{N}\frac{\mathbb{S}_{2,s}^{\frac{N}{2s}}}{\lambda^{\frac{N-2s}{2s}}}-\frac{\lambda}{2_s^*}.$$
\begin{proof}
Clearly for $\lambda<\left(\frac{2_s^*s}{N}\right)^{2s/N}\mathbb{S}_{2,s}$, we have $\left(\frac{s}{N}\frac{\mathbb{S}_{2,s}^{\frac{N}{2s}}}{\lambda^{\frac{N-2s}{2s}}}-\frac{\lambda}{2_s^*}\right)>0$. Consider $\epsilon>0$ to be sufficiently small. Then for any $t\geq0$, 
\begin{align}
I_{\lambda,n}(t\Psi_\epsilon)&=\frac{t^2}{2}\int_{\mathbb{R}^{2N}}\frac{|\Psi_\epsilon(x)-\Psi_\epsilon(y)|^2}{|x-y|^{N+2s}}dxdy+\int_{\Omega}G_n(x,t\Psi_\epsilon)dx-\frac{\lambda}{2_s^*}\int_{\Omega}|t\Psi_\epsilon+w_n|^{2_s^*}dx\nonumber\\&=\frac{t^2}{2}\int_{\mathbb{R}^{2N}}\frac{|\Psi_\epsilon(x)-\Psi_\epsilon(y)|^2}{|x-y|^{N+2s}}dxdy+\int_{\Omega}\frac{|t\Psi_\epsilon|}{(w_n+1/n)^\gamma}\nonumber\\&~~~-\frac{1}{1-\gamma}\int_{\Omega}(t\Psi_\epsilon+w_n+1/n)^{1-\gamma}-(w_n+1/n)^{1-\gamma}-\frac{\lambda}{2_s^*}\int_{\Omega}|t\Psi_\epsilon+w_n|^{2_s^*}dx\nonumber\\&\leq \frac{t^2}{2}(\mathbb{S}_{2,s}^{N/2s}+a_1\epsilon^{N-2s})+tn^\gamma\int_{\Omega}|\Psi_\epsilon|+\frac{\lambda}{2_s^*}-\frac{\lambda}{2_s^*}\nonumber\\&~~~-\frac{1}{1-\gamma}\int_{\Omega}(t\Psi_\epsilon+w_n+1/n)^{1-\gamma}-(w_n+1/n)^{1-\gamma}-\frac{\lambda t^{2_s^*}}{2_s^*}\int_{\Omega}|\Psi_\epsilon|^{2_s^*}dx\nonumber\\&\leq \frac{t^2}{2}(\mathbb{S}_{2,s}^{N/2s}+a_1\epsilon^{N-2s})+tn^\gamma a_3^{1/q}\epsilon^{(N-2s)/2}+\frac{\lambda}{2_s^*}-\frac{\lambda}{2_s^*}\nonumber\\&~~~-\frac{1}{1-\gamma}\int_{\Omega}(t\Psi_\epsilon+w_n+1/n)^{1-\gamma}-(w_n+1/n)^{1-\gamma}-\frac{\lambda t^{2_s^*}}{2_s^*}(\mathbb{S}_{2,s}^{N/2s}-a_2\epsilon^N).
\end{align}
Assume $\lambda\leq1$ and denote a function $h:\mathbb{R}^+\rightarrow\mathbb{R}$ as follows.
\begin{align}
h(t)&=\frac{\lambda}{2_s^*}-\frac{1}{1-\gamma}\int_{\Omega}(t\Psi_\epsilon+w_n+1/n)^{1-\gamma}-(w_n+1/n)^{1-\gamma}\nonumber\\&\leq \frac{1}{2_s^*}-\frac{1}{1-\gamma}\int_{\Omega}(t\Psi_\epsilon+w_n+1/n)^{1-\gamma}-(w_n+1/n)^{1-\gamma}\nonumber\\&\leq \frac{1}{2_s^*}-\frac{1}{1-\gamma}\int_{\Omega}(t\Psi_\epsilon)^{1-\gamma}+C.
\end{align} 
Since $w_n\in \bar{H}$, we can bound $\int_{\Omega} (w_n+1/n)^{1-\gamma}$ uniformly by a constant $C>0$, independent of $n$. Clearly as $t\rightarrow\infty$, $h(t)\rightarrow -\infty$. Hence, there exists $T>0$ such that for every $t\geq T$, $h(t)\leq0$. Thus, for $t\geq T$ we get
\begin{align}
I_{\lambda,n}(t\Psi_\epsilon)&\leq \frac{t^2}{2}(\mathbb{S}_{2,s}^{N/2s}+a_1\epsilon^{N-2s})+tn^\gamma a_3^{1/q}\epsilon^{(N-2s)/2}-\frac{\lambda t^{2_s^*}}{2_s^*}(\mathbb{S}_{2,s}^{N/2s}-a_2\epsilon^N)-\frac{\lambda}{2_s^*}\nonumber\\&= \bar{h}_\epsilon(t).\nonumber
\end{align}
A simple use of basic calculus yields that the maximum value of $\bar{h}_\epsilon$ is attained at $$t_\lambda=\left(\frac{1}{\lambda}\right)^{\frac{N-2s}{4s}}+o(\epsilon^{(N-2s)/2}).$$
Thus, we have
\begin{align}
I_{\lambda,n}(t\Psi_\epsilon)&\leq \frac{s}{N}\frac{\mathbb{S}_{2,s}^{\frac{N}{2s}}}{\lambda^{\frac{N-2s}{2s}}}+o(\epsilon^{(N-2s)/2})-\frac{\lambda}{2_s^*}\nonumber\\&< \frac{s}{N}\frac{\mathbb{S}_{2,s}^{\frac{N}{2s}}}{\lambda^{\frac{N-2s}{2s}}}-\frac{\lambda}{2_s^*}.
\end{align}
For any $t<T$,
\begin{align}
I_{\lambda,n}(t\Psi_\epsilon)&\leq\frac{t^2}{2}\int_{\mathbb{R}^{2N}}\frac{|\Psi_\epsilon(x)-\Psi_\epsilon(y)|^2}{|x-y|^{N+2s}}dxdy+\int_{\Omega}\frac{|t\Psi_\epsilon|}{(w_n+1/n)^\gamma}\nonumber\\&\leq \frac{t^2}{2}(\mathbb{S}_{2,s}^{N/2s}+a_1\epsilon^{N-2s})+tn^\gamma a_3^{1/q}\epsilon^{(N-2s)/2}\nonumber\\&<\frac{T^2}{2}(\mathbb{S}_{2,s}^{N/2s}+a_1\epsilon^{N-2s})+Tn^\gamma a_3^{1/q}\epsilon^{(N-2s)/2}.\nonumber
\end{align}
Choose $\lambda^*>0$ depending on $T,N,2s,\mathbb{S}_{2,s}$ such that for $\lambda\in(0,\lambda^*)$ we obtain 
$$I_{\lambda,n}(t\Psi_\epsilon)<\frac{s}{N}\frac{\mathbb{S}_{2,s}^{\frac{N}{2s}}}{\lambda^{\frac{N-2s}{2s}}}-\frac{\lambda}{2_s^*}.$$
Denote $\Lambda=\min\{1,\left(\frac{s}{N}2_s^*\mathbb{S}_{2,s}\right)^{2s/N},\lambda^*\}$. Then for $0<\lambda<\Lambda$ we have $$\sup\{I_{\lambda,n}(t\Psi_\epsilon): t\geq 0\}<\frac{s}{N}\frac{\mathbb{S}_{2,s}^{\frac{N}{2s}}}{\lambda^{\frac{N-2s}{2s}}}-\frac{\lambda}{2_s^*.}$$ Hence the result.
\end{proof}
\end{lemma}
\begin{proof}[Proof of Theorem $\ref{second solution}$]
We at first need to produce a Palais-Smale sequence named $(v_{n,m})$ of $I_{\lambda,n}$ in $H_n$ using the Ekeland variational principle (see Theorem $\ref{eke}$). Then by  Lemmata $\ref{P-S}$ and $\ref{upper bound}$, there exists $\Lambda>0$ such that for any $\lambda\in (0,\Lambda)$, the functional $I_{\lambda,n}$ satisfies the Palais-Smale compactness condition in $H_n$. This guarantees the existence of a critical point $v_n$ of $I_{\lambda,n}$ in $H_n$ for any $\lambda\in (0,\Lambda)$.\\
Observe that $H_n\subset H_0^s(\Omega)$ is a complete Hilbert manifold. Since the functional $I_{\lambda,n}$ is $C^1$ and bounded from below on $H_n$, we denote $k_n=\underset{v\in H_n}{\inf}I_{\lambda,n}(v)$. Hence, there exists a sequence $(u_{n,m})\subset H_n$ such that $I_{\lambda,n}(u_{n,m})\rightarrow k_n$ as $m\rightarrow\infty$ and for every $\epsilon>0$ there exists $m_0\in \mathbb{N}$ such that $I_{\lambda,n}(u_{n,m})<k_n+\epsilon$ for every $m\geq m_0$. The functional $I_{\lambda,n}$ satisfies the hypotheses of Ekeland variational principle stated in Theorem $\ref{eke}$. By choosing $\delta=\sqrt{\epsilon}$ in Theorem $\ref{eke}$, we guarantee the existence of a sequence $(v_{n,m})\subset H_n$ such that $(I_{\lambda,n}(v_{n,m}))$ is uniformly bounded and $ I_{\lambda,n}^\prime(v_{n,m})\rightarrow 0$. This implies $(v_{n,m})$ is a Palais-Smale sequence and we conclude our proof.
\end{proof}
\section{Existence of SOLA to $(P_\lambda)$}\label{sec 3}
Let us denote $u_n=v_n+w_n$, where $w_n\in \bar{H}$ and $v_n\in H_n$ are the positive weak solutions to $(P^1_{\lambda,n})$ and $(P^2_{\lambda,n})$, respectively. Thus, $u_n\in H$ is a positive weak solution to $(P_{\lambda,n})$ and for every $\omega\subset\subset\Omega$, there exists a constant $C_\omega$ such that $u_n\geq C_\omega>0$. Refer Section $\ref{2}$ for these notations of the function spaces. \\
In this section, we discuss the boundedness of the sequence of solutions $(u_n)$ in a suitable fractional Sobolev space and also prove the existence of SOLA to $(P_\lambda)$.
\begin{lemma}\label{lemma1}
	Let $0<\gamma<1$ and $u_n\in H$ be a weak solution to $(P_{\lambda,n})$ as given in Theorem $\ref{approx solution}$. Then the sequence $(u_n)$ is bounded in $W_0^{s_1,q}(\Omega)$ for every $0<s_1<s$ and $1\leq q<\frac{N}{N-s}$.
	\begin{proof}
We follow Panda et al. in \cite{Panda} to prove this lemma. Let $u_n\in H$ be a weak solution of $(P_{\lambda,n})$. Then for any $k\geq1$, consider $\phi=T_k(u_n)$ as a test function in the weak formulation $\eqref{approximating weak}$ of $(P_{\lambda,n})$ and we get
	\begin{align}
		\int_{\mathbb{R}^N}|(-\Delta)^{s/2}T_k(u_n)|^2&\leq\int_{\mathbb{R}^{N}}(-\Delta)^{s/2}u_n\cdot(-\Delta)^{s/2}T_k(u_n)\nonumber\\&=\int_{\Omega}\frac{1}{(u_n+\frac{1}{n})^\gamma}T_k(u_n)+\int_{\Omega}\lambda u_n^{2_s^*-1}T_k(u_n)+\int_{\Omega}\mu_n T_k(u_n).
\end{align}
Clearly,
 $\frac{T_k(u_n)}{(u_n+\frac{1}{n})^\gamma}\leq \frac{u_n}{(u_n+\frac{1}{n})^\gamma}\leq u_n^{1-\gamma}$. Since  $\|u_n\|_{L^{2_s^*}(\Omega)}=1$ and the sequence $(\mu_n)$ is $L^1$ bounded we have
\begin{align}\label{4p6}
	\int_{\mathbb{R}^N}|(-\Delta)^{s/2}T_k(u_n)|^2&\leq \int_{\Omega} u_n^{1-\gamma}+k\|\mu_n\|_{L^1(\Omega)}+\lambda k\|u_n\|_{L^{2_s^*-1}(\Omega)}^{2_s^*-1}\nonumber\\&\leq C_1\|u_n\|_{L^{2_s^*}(\Omega)}^{1-\gamma}+C_2k+\lambda C_3\|u_n\|_{L^{2_s^*}(\Omega)}^{2_s^*-1}\nonumber\\&\leq Ck.
\end{align}
Thus, $(T_k(u_n))$ is bounded in $H_0^s(\Omega)$. Consider the sets $I=\{x\in \Omega: |(-\Delta)^{s/2}u_n|\geq t \}$, $I_1=\{x\in \Omega: |(-\Delta)^{s/2}u_n|\geq t, u_n\leq k \}$ and $I_2=\{x\in \Omega: u_n>k\}$. Then, $I\subset I_1\cup I_2$, which implies $m(I)\leq m(I_1)+m(I_2)$, where $m$ is the Lebesgue measure. Using the Sobolev inequality, stated in Theorem $\ref{constant}$, we establish
\begin{align}\label{4p7}
\left(\int_{\Omega}|T_k(u_n)|^{2_s^*}\right)^{\frac{2}{2_s^*}}&\leq C^\prime \int_{\mathbb{R}^N}|(-\Delta)^{s/2}T_k(u_n)|^2\nonumber\\&\leq Ck.
\end{align}
Then on $I_2$, the equation $\eqref{4p7}$ becomes 
\begin{align}\label{4p8}
k^2 m(I_2)^{\frac{2}{2_s^*}}&\leq Ck\nonumber\\ m(I_2)& \leq \frac{ C}{k^{\frac{N}{N-2s}}}, ~\forall k\geq 1.
\end{align}
This proves $(u_n)$ is bounded in ${M}^{\frac{N}{N-2s}}(\Omega)$. Similarly on $I_1$, the equation $\eqref{4p6}$ becomes 
\begin{align}\label{4p9}
t^2 m(I_1)&\leq Ck\nonumber\\ m(I_1)&\leq  \frac{Ck}{t^2}, ~\forall k>1.
\end{align} 
\noindent On combining $\eqref{4p8}$ and $\eqref{4p9}$ we have 
\begin{align}
m(I)\leq \frac{Ck}{t^2}+\frac{C}{k^{\frac{N}{N-2s}}}, ~\forall k>1.\nonumber
\end{align}
Choose $k=t^{\frac{N-2s}{N-s}}$. Thus, we obtain $$m(I)\leq \frac{C}{t^{\frac{N}{N-s}}},~\forall t\geq 1.$$
Therefore, the sequence $((-\Delta)^{s/2}u_n)$ is bounded in $M^{\frac{N}{N-s}}(\Omega)$. By the continuous embedding $\eqref{continuous}$, the sequences $(u_n)$ and $((-\Delta)^{s/2}u_n)$ are bounded in $L^q(\Omega)$ for every $q<\frac{N}{N-s}$. From Definition $\ref{Bessel}$ and Theorem $\ref{properties}$, we conclude that $(u_n)$ is uniformly bounded in $L^{s,q}(\mathbb{R}^N)$  and hence bounded in $W_0^{s_1,q}(\Omega)$, for every $s_1<s$ and $q<\frac{N}{N-s}$.
	\end{proof}
\end{lemma}
	\noindent We now pass the limit $n\rightarrow\infty$ in the weak formulation $\eqref{approximating weak}$ to prove the existence of a SOLA to $(P_\lambda)$.
	\begin{proof}[Proof of Theorem $\ref{main theorem}$]
Let $\mu\in \mathcal{M}(\Omega)$, $0<\gamma< 1$ and $u_n\in H$ be a weak solution of $(P_{\lambda,n})$ for $\lambda\in (0,\Lambda)$. According to Lemma $\ref{lemma1}$, $(u_n)$ is bounded in $W_0^{s_1,q}(\Omega)$,  for every $0<s_1<s$ and $ q<\frac{N}{N-s}$. Hence, there exists $u\in W_0^{s_1,q}(\Omega)$ such that $u_n\rightarrow u$ weakly in $W_0^{s_1,q}(\Omega)$. Thus, $u_n\rightarrow u$ a.e. in $\mathbb{R}^N$ and $u\equiv 0$ in $\mathbb{R}^N\setminus\Omega$. \\
Denote $$\Phi(x,y)=\phi(x)-\phi(y),~ \bar{U}_n(x,y)=u_n(x)-u_n(y),$$ $$\bar{U}(x,y)=u(x)-u(y) ~\text{and}~ d\nu=\frac{dxdy}{|x-y|^{N+2s}}.$$
For every $\phi\in C_c^\infty(\Omega)$, from the weak formulation of $(P_{\lambda,n})$ we have 
\begin{align}
\int_{\mathbb{R}^{2N}} \bar{U}_n(x,y)\Phi(x,y)d\nu&=\int_{\Omega}\frac{1}{(u_n+\frac{1}{n})^\gamma}\phi+\int_{\Omega}\lambda u_n^{2_s^*-1}\phi+\int_{\Omega}\mu_n \phi.\nonumber
\end{align}
We can rewrite the above equation as
\begin{align}\label{4p10}
 \int_{\mathbb{R}^{2N}} \bar{U}(x,y)\Phi(x,y)d\nu+\int_{\mathbb{R}^{2N}} (\bar{U}_n(x,y)-\bar{U}(x,y))\Phi(x,y)d\nu= \int_{\Omega}\frac{1}{(u_n+\frac{1}{n})^\gamma}\phi+\int_{\Omega}\lambda u_n^{2_s^*-1}\phi+\int_{\Omega}\mu_n \phi.
\end{align}
Clearly, $$\lim\limits_{n\rightarrow\infty}\int_{\Omega}\mu_n\phi=\int_{\Omega}\phi d\mu,$$
$$\lim\limits_{n\rightarrow\infty}\lambda\int_{\Omega}u_n^{2_s^*-1}\phi=\lambda\int_{\Omega}u^{2_s^*-1}\phi.$$
On using the Dominated Convergence Theorem we get
$$\lim\limits_{n\rightarrow\infty}\int_{\Omega}\frac{1}{(u_n+\frac{1}{n})^\gamma}\phi=\int_{\Omega}\frac{1}{u^\gamma}\phi.$$
Now the integral on the left hand side of $\eqref{4p10}$ can be expressed as follows.
\begin{align}
\int_{\mathbb{R}^{2N}} (\bar{U}_n(x,y)-\bar{U}(x,y))\Phi(x,y)d\nu&=\int_{\Omega\times\Omega} (\bar{U}_n(x,y)-\bar{U}(x,y))\Phi(x,y)d\nu\nonumber\\&~~+\int_{\Omega\times(\mathbb{R}^{N}\setminus\Omega)} (\bar{U}_n(x,y)-\bar{U}(x,y))\Phi(x,y)d\nu\nonumber\\&~~+\int_{(\mathbb{R}^{N}\setminus\Omega)\times\Omega} (\bar{U}_n(x,y)-\bar{U}(x,y))\Phi(x,y)d\nu\nonumber\\&=J_{1,n}+J_{2,n}+J_{3,n}.\nonumber
\end{align}
Observe $\bar{U}_n\rightarrow \bar{U}$ a.e. in $\mathbb{R}^N$. Since $\Omega$ is bounded, using Lemma $\ref{lemma1}$ and Vitali's lemma we have $\bar{U}_n\rightarrow \bar{U}$ strongly in $L^1(\Omega\times \Omega, d\nu)$. Hence, $J_{1,n}\rightarrow 0$ as $n\rightarrow\infty$. \\
Let $(x,y)\in \Omega\times(B_R\setminus\Omega)$, then $$\underset{(x,y)\in \Omega\times(B_R\setminus\Omega)}{\sup}\frac{1}{|x-y|^{N+2s}}\leq C<\infty.$$ Hence, by the Dominated Convergence Theorem, $J_{2,n}\rightarrow 0$ as  $n\rightarrow\infty$ and similarly $J_{3,n}\rightarrow 0$ as $n\rightarrow \infty$. Thus on passing the limit $n\rightarrow \infty$ in $\eqref{4p10}$, we obtain 
\begin{align}
\int_{\mathbb{R}^{2N}} \bar{U}(x,y)\Phi(x,y)d\nu= \int_{\Omega}\frac{1}{u^\gamma}\phi+\int_{\Omega}\lambda u^{2_s^*-1}\phi+\int_{\Omega} \phi d\mu.
\end{align}
Thus, $u$ is a positive SOLA to $(P_\lambda)$, in the sense of Definition $\ref{SOLA}$.
\end{proof}
	\section*{Appendix}\label{sec 4}
	We now discuss about the multiplicity of solutions to the problem $(P_\lambda)$ for the case $0<\gamma<1$. Let $I_\lambda$ be the corresponding energy functional of $(P_\lambda)$ given by
	$$I_\lambda(u)=\frac{1}{2}\int_{\mathbb{R}^{2N}}\frac{|u(x)-u(y)|^2}{|x-y|^{N+2s}}dxdy-\frac{1}{1-\gamma}\int_{\Omega}u^{1-\gamma}-\frac{\lambda}{2_s^*}\int_{\Omega}u^{2_s^*}-\int_{\Omega}ud\mu.$$
For $0<\gamma<1$, $I_\lambda$ is a $C^1$ functional and $I_\lambda(0)=0.$ From Theorem $\ref{main theorem}$, $u\in W_0^{s_1,q}(\Omega)$ is a  positive SOLA to $(P_\lambda)$ for every $s_1<s$, $q<\frac{N}{N-s}$. Then $u$ is also a Nehari solution of $(P_\lambda)$, i.e. $u\in N_\lambda=\{u\in W_0^{s_1,q}(\Omega):\langle I_\lambda^\prime(u),u\rangle=0 \}$. Here
$$\langle I_\lambda^\prime(u),u\rangle=\int_{\mathbb{R}^{2N}}\frac{|u(x)-u(y)|^2}{|x-y|^{N+2s}}dxdy-\int_{\Omega}u^{1-\gamma}-\lambda\int_{\Omega}u^{2_s^*}-\int_{\Omega}ud\mu.$$
Consider the fibre map $\psi:(0,\infty)\rightarrow\mathbb{R}$ defined by 
$$\psi(t)=\frac{t^2}{2}\int_{\mathbb{R}^{2N}}\frac{|u(x)-u(y)|^2}{|x-y|^{N+2s}}dxdy-\frac{t^{1-\gamma}}{1-\gamma}\int_{\Omega}u^{1-\gamma}-\frac{\lambda t^{2_s^*}}{2_s^*}\int_{\Omega}u^{2_s^*}-t\int_{\Omega}ud\mu.$$
Then $$\psi^\prime(t)=At-Bt^{-\gamma}-\lambda C t^{2_s^*-1}-D$$ and $$\psi^{\prime\prime}(t)= A+\gamma Bt^{-\gamma-1}-(2_s^*-1)\lambda Ct^{2_s^*-2}$$
where $A=\int_{\mathbb{R}^{2N}}\frac{|u(x)-u(y)|^2}{|x-y|^{N+2s}}dxdy$, $B=\int_{\Omega}u^{1-\gamma}$, $C=\int_{\Omega}u^{2_s^*}$ and $D=\int_{\Omega}ud\mu$. Since $u\in N_\lambda$, $\psi^\prime(1)=A-B-\lambda C-D=0.$ Clearly $\psi^\prime(t)\rightarrow -\infty$ as $t\rightarrow 0$ and $t\rightarrow \infty$.\\
{\it Case 1:} If $\Lambda>\lambda>\frac{A+\gamma B}{(2^*_s-1)C}$, then $\psi^{\prime\prime}(1)<0.$ Hence, there exists at least one Nehari solution to $(P_\lambda)$.\\
{\it Case 2:} If $\lambda<\min \{\frac{A+\gamma B}{(2^*_s-1)C},\Lambda\}$, then $\psi^{\prime\prime}(1)>0$ and we guarantee the existence of at least three nontrivial Nehari solution to $(P_\lambda)$.\\
{\it Case 3:} If $\lambda=\frac{A+\gamma B}{(2^*_s-1)C}$, then $\psi^{\prime\prime}(1)=0$. Thus, we obtain a saddle point and hence there exists at least one Nehari solution to $(P_\lambda)$.
	\section*{Acknowledgement}
The author Debajyoti Choudhuri thanks the grant received from Science and Engineering Research Board (SERB) for the project MATRICS- MTR/2018/000525. Akasmika Panda thanks the financial assistantship received from the Ministry of Human Resource Development (M.H.R.D.), Govt. of India. Ratan Kr. Giri acknowledges the financial support and facilities
received from the Mathematics Department, Technion - Israel Institute of Technology, Haifa. The authors thank the anonymous reviewers for their comments and suggestions.
 
\end{document}